\documentclass[12pt]{amsart}
\usepackage{amssymb}
\usepackage{amsmath, amscd}
\usepackage{amsthm}
\usepackage{amsmath}
\usepackage{comment}
\usepackage[table,xcdraw]{xcolor}
\usepackage[colorlinks=true, linkcolor=blue,urlcolor=blue]{hyperref}

\newtheorem{theorem}{Theorem}[section]

\newtheorem{proposition}[theorem]{Proposition}
\newtheorem{lemma}[theorem]{Lemma}

\newtheorem{corollary}[theorem]{Corollary}
\theoremstyle{definition}

\newtheorem{example}[theorem]{Example}
\newtheorem{definition}[theorem]{Definition}

\newcommand{\bigzero}{\mbox{\normalfont\Large\bfseries 0}}

\begin{document}
	
\author[P. Danchev]{Peter Danchev}
\address{Institute of Mathematics and Informatics, Bulgarian Academy of Sciences, 1113 Sofia, Bulgaria}
\email{danchev@math.bas.bg; pvdanchev@yahoo.com}	
	
\author[O. Hasanzadeh]{Omid Hasanzadeh}
\address{Department of Mathematics, Tarbiat Modares University, 14115-111 Tehran Jalal AleAhmad Nasr, Iran}
\email{o.hasanzade@modares.ac.ir; hasanzadeomiid@gmail.com}

\author[A. Moussavi]{Ahmad Moussavi}
\address{Department of Mathematics, Tarbiat Modares University, 14115-111 Tehran Jalal AleAhmad Nasr, Iran}
\email{moussavi.a@modares.ac.ir; moussavi.a@gmail.com}

\author[M. Esfandiar]{Mehrdad Esfandiar}
\address{Department of Mathematics and Computer Science Shahed University Tehran, Iran}
\email{mehrdad.esfandiar@shahed.ac.ir}

\title[Weakly $\Delta$U Rings]{A Generalization of $\Delta$U Rings}

\maketitle

\date{\today}

\begin{abstract}
In this paper, we introduce and study a new class of rings calling them {\it weakly $\Delta U$-rings}, hereafter abbreviated as {\it $W\Delta U$-rings} for short. A ring $R$ is said to be $W\Delta U$ if every unit of $R$ can be expressed as $\pm 1 + d$ for some $d \in \Delta(R)$, where $\Delta(R)$ is the largest Jacobson radical of $R$ that is closed under multiplication by units.

Utilizing the known structure of $\Delta(R)$, we investigate the relationships between $W\Delta U$ rings and certain classical concepts such as $\Delta U$-rings, $UJ$-rings, $WUJ$-rings, as well as clean and exchange rings. Among the main results, we show that a matrix ring $M_n(R)$ is never $W\Delta U$ for any $n \ge 2$. We also provide complete characterizations of local, semi-local, semi-simple and semi-regular rings that are $W\Delta U$. Furthermore, it is shown for exchange rings that the $W\Delta U$ property is equivalent to being $WUJ$.

Furthermore, the behavior of $W\Delta U$-rings under various ring extensions, including skew polynomial rings, skew power series rings, triangular matrix rings, trivial extensions and group rings, is thoroughly examined. Several examples are given to illustrate that the class of $W\Delta U$-rings properly contains the class of $\Delta U$-rings. Finally, necessary and sufficient conditions for a group ring $RG$ to be $W\Delta U$ are established too.	

Resuming all of the presented above, our results expanded those by Karaba\c{c}ak et al. published in J. Algebra \& Appl. (2021).
\end{abstract}

\section{Introduction and Basic Notions}

Throughout this paper, \(R\) is always a ring with identity as it does {\it not} need to be commutative. In our work, we will use the following pretty standard symbols:

\medskip

\begin{itemize}
\item \(U(R)\): the set of units (= invertible elements) of \(R\);
\item \(Nil(R)\): the set of nilpotent elements of \(R\);
\item \(C(R)\): the center of \(R\);
\item \(Id(R)\): the set of idempotent elements of \(R\);
\item \(J(R)\): the Jacobson radical of \(R\).
\end{itemize}

\medskip

Likewise, \(M_n(R)\) stands for the ring of \(n \times n\) matrices over \(R\), and \(T_n(R)\) stands for the ring of \(n \times n\) upper triangular matrices over \(R\). A ring is called {\it abelian} if every its idempotent lies in the center, meaning that \(\operatorname{Id}(R) \subseteq C(R)\).
	
Now, let us recall some classical things that we will use explicitly henceforth: A ring \(R\) is said to be {\it Boolean} if each its element is idempotent. A ring \(R\) is called {\it regular} (that means {\it von Neumann regular}) if, for each \(a\in R\), there is some \(x\in R\) such that \(axa=a\). If, however, \(x\) is a unit, then \(R\) is termed {\it unit-regular}. Also, \(R\) is {\it strongly regular} if, for each \(a\in R\), we have \(a\in a^2R\).

In another vein, mimicking \cite{8}, a ring \(R\) is {\it exchange} if, for every \(a\in R\), there is an idempotent \(e^2=e\in aR\) such that \(1-e\in (1-a)R\). Besides, a ring \(R\) is {\it clean} if every element is the sum of an idempotent and a unit. It is well known that any clean ring is exchange, but the converse is not generally true. However, for abelian rings, this holds always.

Additionally, a ring \(R\) is {\it semi-regular} if \(R/J(R)\) is regular and all idempotents lift modulo \(J(R)\). So, semi-regular rings are exchange, but again the reverse is not true in general.
	
On another track, a ring \(R\) is named a {\it \(UU\)-ring} if \(U(R)=1+\operatorname{Nil}(R)\) (see \cite{12}). Later, Danchev individually defined in \cite{wuu} the so-called {\it weakly UU-rings} (or, shortly, just {\it WUU-rings}) via the condition \(U(R) = \operatorname{Nil}(R) \pm 1\). He showed that an exchange ring is WUU if and only if it is clean and WUU, if and only if it is strongly weakly nil-clean, if and only if \(J(R)\) is nil and \(R/J(R)\) is isomorphic to a Boolean ring \(B\), or to \(\mathbb{Z}_3\), or to \(B \times \mathbb{Z}_3\). He also proved that, for \(n \ge 2\), the matrix ring \(M_n(R)\) is never a WUU-ring.
	
Since it is always fulfilled that \(\pm 1 + J(R) \subseteq U(R)\), it is natural to look at rings in which an equality holds. In this way, a ring $R$ for which \(U(R) = 1 + J(R)\) is called a {\it \(JU\)-ring} (or simply a {\it \(UJ\)-ring}) (see, respectively, \cite{D} and \cite{14}). Later on, Danchev generalized in \cite{wuj} this notion to {\it weakly \(JU\)-rings} (or simply {\it WJU-rings}). He proved that an exchange ring \(R\) is WJU if and only if it is a clean WJU-ring, if and only if it is weakly semi-boolean, if and only if idempotents lift modulo \(J(R)\) and \(R/J(R)\) is isomorphic to a Boolean ring \(B\), or to \(\mathbb{Z}_3\), or to \(B \times \mathbb{Z}_3\).
	
On the other side, imitating Chen \cite{16}, an element is called {\it \(J\)-clean} if it can be written as an idempotent plus an element of the Jacobson radical, and a ring is {\it \(J\)-clean} if every element is \(J\)-clean. For such a ring $R$ this is equivalent to \(R/J(R)\) being Boolean and idempotents lifting modulo \(J(R)\) (such rings are also known as {\it semi-boolean}). Moreover, it was shown in \cite{16} that \(R\) is \(J\)-clean if and only if \(R\) is a clean \(UJ\)-ring.
	
The main working tool of this article is the set \(\Delta(R)\). It appeared in an exercise proposed by Lam \cite{23, lam} and was later studied by Leroy and Matczuk in \cite{2}. We know that the subring \(\Delta(R)\) is the largest Jacobson radical of \(R\) which is closed under multiplication by units as well as that \(J(R) \subseteq \Delta(R)\) and \(\Delta(R)=J(T)\), where \(T\) is the subring generated by the units of \(R\). The equality \(\Delta(R)=J(R)\) is valid precisely when \(\Delta(R)\) is an ideal of \(R\).
	
In 2021, Karaba\c{c}ak and his coauthors introduced in \cite{7} a new family of rings as a nontrivial generalization of \(UJ\)-rings. They called these {\it \(\Delta U\)-rings}: A ring \(R\) is a \(\Delta U\)-ring provided that \(1+\Delta(R)=U(R)\). They also defined {\it \(\Delta\)-clean} rings as those rings for which each element is the sum of an idempotent and an element from \(\Delta(R)\). Thereby, any \(\Delta\)-clean ring is clean, but the reciprocal implication is definitely untrue. They also showed that \(R\) is \(\Delta\)-clean if, and only if, \(R\) is a clean \(\Delta U\)-ring.
	
Motivated by all of this, as a natural extension of these ideas, we presently introduce the concept of {\it weakly \(\Delta U\)-rings} saying that a ring \(R\) is weakly \(\Delta U\) (or just a {\it $W\Delta U$-ring} for simpleness of the exposition) if every unit is either the sum or the difference of an idempotent and an element from \(\Delta(R)\). Equivalently, for each unit \(u\in U(R)\), we have \(u = \pm 1 + d\) for some \(d\in \Delta(R)\). Clearly, any \(\Delta U\)-ring is a $W\Delta U$-ring, but it will be demonstrated in the sequel that this implication is invalid in all generality.
	
Specifically, our plan how to extend the previously mentioned results is like this: In the next Section 2, we discover some background properties of being a $W\Delta U$-ring by establishing, e.g., Theorems~\ref{lad} and Proposition~\ref{mino}, respectively. In the subsequent Section 3, we achieve some results pertaining to the description of $W\Delta U$-rings -- e.g., Theorems~\ref{9}, \ref{main}, \ref{3.16} and \ref{fin2}, respectively. In the final Section 4, we deal with certain extensions of $W\Delta U$-rings such as matrix and group rings ones -- e.g., Theorem~\ref{3.29} and Proposition~\ref{torsion2}.
	
\section{Fundamental Properties of $W\Delta U$-Rings}

Now, for our further purposes, let \(R\) be a ring and let \(\alpha: R \to R\) be a ring homomorphism. The {\it skew polynomial ring} \(R[x; \alpha]\) consists of all polynomials in \(x\) with coefficients from \(R\) and with multiplication rule \(x r = \alpha(r) x\) for all \(r\in R\). If \(\alpha = 1_R\), we then just get the usual polynomial ring \(R[x]\).
	
Similarly, the {\it skew formal power series ring} \(R[[x; \alpha]]\) consists of all formal power series \(\sum_{i=0}^{\infty} r_i x^i\) with \(r_i \in R\) and with the same rule \(x r = \alpha(r) x\). When \(\alpha = 1_R\), this becomes the usual power series ring \(R[[x]]\).
	
Finally, a ring \(R\) is called {\it weakly 2-primal} provided that \(\operatorname{Nil}(R) = \operatorname{L}(R)\), where \(\operatorname{L}(R)\) is the Levitzki radical. For example, any commutative ring or reduced ring is ever weakly 2-primal.

\medskip

We begin our more concrete considerations with the following technicalities.

\begin{proposition}\label{local}
A ring $R$ is local if, and only if, \( R = U(R) \cup \Delta(R) \).
\end{proposition}

\begin{proof}
If \( R \) is local, then \[ R=U(R) \cup J(R) \subseteq U(R) \cup \Delta(R) \subseteq R ,\] and thus \( R=U(R) \cup \Delta(R) \).

Conversely, suppose \( R=U(R) \cup \Delta(R) \) and \( a \notin U(R) \). Then, \( a \in \Delta(R) \), and so \( 1-a \in U(R) \), whence \( R \) is a local ring.
\end{proof}

For a ring \( R \), we say that the set \( Id(R) \) (respectively, \( U(R) \) and \( \Delta(R) \)) is \textit{trivial} whenever \( Id(R) = \{0,1\} \) (respectively, \( U(R) = \{1\} \) and \( \Delta(R) = \{0\}\)).

\medskip

Recall also that a ring $R$ is said to be {\it Dedekind-finite} if $ab=1$ amounts to $ba=1$ for any $a,b\in R$. In other words, all one-sided inverses in such a ring are necessarily two-sided.

\begin{theorem}\label{lad}
Let \( R \) be a ring such that \( R = U(R) \cup \Delta(R) \cup Id(R) \). Then, exactly one of the following claims holds:

(1) \( R \) is a local ring.
	
(2) \( R \) is a Boolean ring.
	
(3) \( R \) is a non-abelian Dedekind-finite ring with characteristic \( 2 \).
\end{theorem}

\begin{proof}	
The proof proceeds by analyzing the following distinct cases:

\medskip
	
\textbf{Case 1.} Assume that $\Delta(R) = \{0\}$. Thus, $R = U(R)\cup Id(R)$. Hence, $R$ is either a local ring or a Boolean ring in accordance with \cite[Theorem 2.4]{cc}.

\medskip
	
\textbf{Case 2.} Suppose $\Delta(R) \neq \{0\}$. Selecting a nonzero element $d \in \Delta(R)$, we can write $u := 1 + d \in U(R)$, so that $U(R)$ is non-trivial.
	
\textbf{Subcase 2.1.} If $Id(R)$ is trivial, then Proposition \ref{local} applies to get $R = U(R) \cup \Delta(R)$, implying that $R$ is local.
	
\textbf{Subcase 2.2.} Suppose all three subsets $U(R)$, $\Delta(R)$ and $Id(R)$ are non-trivial. Let $e \in Id(R)$ be such that $e \not\in \{0, 1\}$, and let $u := 1 + d \in U(R)$ be as before. We now aim to show that $2 \not\in U(R)$.
	
Assume, for contradiction, that $2 \in U(R)$. Then, obviously $2e \notin U(R)$. If $2e \in Id(R)$, then $(2e)^2 = 2e$ forces $e = 0$, thus contradicting $e \not\in \{0,1\}$. Hence, $2e \in \Delta(R)$. Observe that $2$ lies in the center $C(R)$, and therefore it follows that $e \in \Delta(R)$ (because $\Delta(R)$ is closed with respect to multiplication with respect to unit elements), and hence $e=0$, contradicting again that $e \in Id(R) \setminus \{0,1\}$. Consequently, $2 \in Id(R)$ or $2 \in \Delta(R)$.
	
If $2^2 = 2$, then $2 = 0$, so $char(R) = 2$. If $2 \in \Delta(R)$, then $3 \in U(R)$. But, apparently, $3e \not\in U(R)$ as $e \neq 1$. If $3e \in \Delta(R)$, then we have $e \in \Delta(R) \cap Id(R) = \{0\}$, again a contradiction. Therefore, $3e \in Id(R)$, and $(3e)^2 = 3e$ gives $6e = 0$. Since $3 \in U(R)$, we get $2e = 0$.
	
Now, consider $1 - e$ instead of $e$. The same argument assures $2(1 - e) = 0$. Adding these two equalities yields automatically $2 = 2e + 2(1 - e) = 0$, confirming $char(R) = 2$.
	
To show $R$ is non-abelian, suppose on the contrary that $eu = ue$. Then, $ue \notin U(R)$ and $ue \notin \Delta(R)$. Thus, $(ue)^2 = ue \in Id(R)$, so $ue = e$. Similarly, $u(1 - e) = 1 - e$, whence $u = 1$, contradicting $u \neq 1$. Therefore, $R$ is non-abelian with $char(R) = 2$.
	
Finally, we show that $R$ is Dedekind-finite. To that end, let $a, b \in R$ with $ab = 1$. If $a \in \Delta(R)$, then so is $b$. Indeed, $b \not\in U(R)$ as for otherwise $a = b^{-1} \in U(R)$, and if $b \in Id(R)$, then $1 - b = ab(1 - b) = 0$ ensures $b = 1$, so that $a = 1 \in U(R)$, a contradiction. Since $char(R) = 2$, we define $v := (1 + a)(1 + b) = a + b \in U(R)$.
	
Furthermore, multiplying both sides by $b$ insures $b(a + b) = bv$. Since $bv \not\in U(R)$, and since it cannot lie in $Id(R)$ as that would lead to $b^3 = 0$ contradicting $ab = 1$, we conclude $bv, vb \in \Delta(R)$. However, $vb = (a + b)b = 1 + b^2 \in U(R)$ since $b \in \Delta(R)$, again a contradiction. Thus, $a \not\in \Delta(R)$.
	
If now $a \in U(R)$, we are done. If $a \in Id(R)$, then it must be that $$1 - a = (1 - a)(ab) = (a - a^2)b = 0,$$ and hence $a = 1$ and $ba = 1 = b$. In conclusion, $R$ is Dedekind-finite, as wanted.
\end{proof}

It is pretty evident that both local rings and Boolean rings are abelian and, consequently, they are Dedekind-finite. Hence, the following result can be derived immediately.

\begin{corollary}
If \( R = U(R) \cup \Delta(R) \cup Id(R) \), then $R$ is a Dedekind-finite ring.
\end{corollary}

We now proceed by stating the following.

\begin{definition}
A ring \( R \) is referred to as {\it weakly $\Delta$U}, abbreviated in what follows by {\it W$\Delta$U}, if the group of units satisfies
$$
U(R) = \pm1 + \Delta(R).
$$
This condition is equivalent to requiring that each unit in \( R \) can be expressed as either \( d + 1 \) or \( d - 1 \), where \( d \in \Delta(R) \).	
\end{definition}

\begin{example}
(1) The ring \( \mathbb{Z} \) of all integers is W$\Delta$U, but {\it not} $\Delta$U, because
$
U(\mathbb{Z}) = \{-1, 1\}.
$

(2) The finite ring $\mathbb{Z}_3$ is W$\Delta U$, but {\it not} $\Delta U$.
\end{example}

\begin{example}
Manifestly, every \( WJU \) ring is also a \( W\Delta U \)-ring. However, the converse does {\it not} hold. To see that, consider as done in \cite{12} the ring
$
R = \mathbb{F}_2\langle x, y \rangle / \langle x^2 \rangle.
$
In this ring, we have \( J(R) = \{0\} \), and also that
\[
\mathbb{F}_2 x + xRx \subseteq \Delta(R); \quad U(R) = 1 + \mathbb{F}_2 x + xRx.
\]
It therefore follows that \( U(R) = 1 + \Delta(R) \), and hence \( R \) is a \( \Delta U \)-ring and so a $W\Delta U$-ring. Nevertheless, one observes that \( R \) need not be a \(WJU \)-ring, as expected.
\end{example}

\begin{proposition}\label{1}
Let \(I \subseteq J(R)\) be an ideal of a ring \(R\). Then, \(R\) is W$\Delta$U if, and only if, so is \(R/I\).
\end{proposition}

\begin{proof}
Suppose that \(R\) is a W$\Delta$U ring and choose \(u + I \in U(R/I)\). Thus, \(u \in U(R)\) whence \(u = \pm1 + d\), where \(d \in \Delta(R)\). Therefore, \[(u + I) = \pm(1+I)+(d+I),\] where \(d + I \in \Delta(R)/I = \Delta(R/I)\) in virtue of \cite[Proposition 6]{2}.
	
Conversely, suppose \(R/I\) is a W$\Delta$U ring and \(u \in U(R)\). Thus, \(u + I \in U(R/I)\) whence \[(u + I) = \pm(1 + I) + (d + I),\] where \(d + I \in \Delta(R/I)\). This means that \(u + I = \pm(1 + d) + I\) and so \[u \pm(1 + d) \in I \subseteq J(R) \subseteq \Delta(R).\] Consequently, \(u = \pm1 + d^\prime\), where \(d^\prime \in \Delta(R)\). That is why, \(R\) is a W$\Delta$U ring, as suspected.
\end{proof}

The next reduction consequence is useful.

\begin{corollary}\label{2}
A ring \(R\) is W$\Delta$U if, and only if, \(R/J(R)\) is W$\Delta$U.
\end{corollary}

Another type of technical assertions is as follows.

\begin{proposition}\label{product}
Let $R$ be a W$\Delta$U ring, and $S$ a $\Delta$U ring. Then, $R\times S$ is a W$\Delta$U ring.
\end{proposition}

\begin{proof}
Given \((u, v) \in U(R \times S) = U(R) \times U(S)\), there exists \( d \in \Delta(R) \) such that either \( u = 1 + d \) or \( u = -1 + d \). We consider the following two cases:

\medskip

\noindent\textbf{Case I.} Suppose \( u = 1 + d \). Then, since \( v \in U(S) \), there exists \( d' \in \Delta(S) \) such that \( v = 1 + d' \). Hence,
\[
(u, v) = (1, 1) + (d, d'),
\]
where \( (d, d') \in \Delta(R \times S) \) in view of \cite[Lemma 1(5)]{2}.

\medskip

\noindent\textbf{Case II.} Suppose \( u = -1 + d \). Then, since \( v \in U(S) \), there exists \( d' \in \Delta(S) \) such that \( -v = 1 + d' \), i.e., \( v = -1 + (-d') \). Therefore,
\[
(u, v) = -(1, 1) + (d, -d'),
\]
where \( (d, -d') \in \Delta(R \times S) \) by usage of the same lemma from \cite{2}.
\end{proof}

\begin{proposition}\label{productt}
Let $\{R_i\}$ be a family of rings. Then, the direct product $R=\prod R_i$ of rings $R_i$ is W$\Delta$U if, and only if, each $R_i$ is W$\Delta$U and at most one of them is {\it not} $\Delta$U.
\end{proposition}

\begin{proof}
($\Rightarrow$) Obviously, each $R_i$ is W$\Delta$U. Suppose now $R_{i_1}$ and $R_{i_2}$ $(i_1 \neq i_2)$ are {\it not} $\Delta$U. Then, there are two elements $u_{i_j} \in U (R_{i_j})$ with $j=1,2$ such that $u_{i_1} \in U (R_{i_1})$ and $-u_{i_2} \in U (R_{i_2})$ are both {\it not} $\Delta$U decomposition. Choosing $u=(u_i)$, where $u_i=1$ whenever $i \neq i_{j}\quad (j=1,2)$, we deduce that $u$ and $-u$ are {\it not} the sum of 1 and an element from $\Delta$, as required to get a contradiction. Consequently, each $R_i$ is a W$\Delta$U ring as at most one of them is {\it not} $\Delta$U.

($\Leftarrow$) Assume that $R_{i_0}$ is a W$\Delta$U ring and all of the other members $R_i$ are $\Delta$U. So, with \cite[Proposition 2.4(6)]{7} in mind, we have that $\prod_{i \neq i_{0}} R_i$ is $\Delta$U. Hence, owing to Proposition \ref{product}, we conclude that $R$ is a W$\Delta$U ring.
\end{proof}

Thereby, we arrive at three more direct consequences.

\begin{example}
The ring $\mathbb{Z}_{3}$ is W$\Delta$U, but the direct product $\mathbb{Z}_{3}\times \mathbb{Z}_{3}$ is {\it not} W$\Delta$U.
\end{example}

\begin{corollary}
Let $L=\prod_{i \in I} R_i$ be the direct product of rings $R_i \cong R$ and $|I| \geq 2$. Then, $L$ is a W$\Delta$U ring if, and only if, $L$ is a $\Delta$U ring if, and only if, $R$ is a $\Delta$U ring.
\end{corollary}

\begin{corollary}\label{10}
For any $n \geq 2$, the ring $R^n$ is W$\Delta$U if, and only if, $R^n$ is $\Delta$U if, and only if, $R$ is $\Delta$U.
\end{corollary}

We now continue by proving the following statements. 

\begin{proposition}\label{corner}
For any \( e \in \text{Id}(R) \), if the ring \( R \) is W$\Delta$U, then the corner ring \( eRe \) is also W$\Delta$U.
\end{proposition}

\begin{proof}
Letting \(u \in U(eRe)\), we have \(u + (1-e) \in U(R)\). Under validity of the stated hypothesis, one writes that \[(u + (1-e))= \pm1 + d \in \pm1 + \Delta(R).\] Thus, \(u -  e \in \Delta(R)\) or $u-e+2\in \Delta(R)$.
	
Now, we need to show that \(u -  e \in \Delta(eRe)\) or $u-e+2\in \Delta(eRe)$. The first claim follows with the aid of \cite[Proposition 2.6]{7}. For the second claim, we have to show that $u+2-e\in \Delta(eRe)$. To that target, let \(v\) be an arbitrary unit of \(eRe\) and hence one infers that \(v + 1 - e \in U(R)\). Note also that \(u + 2-e \in \Delta(R)\) implies \(u + 2-e + v + 1 - e \in U(R)\) under presence of the definition of \(\Delta(R)\). Now, taking \(u + v + 3-2e  = t \in U(R)\), one can inspect that \[ e t = t e = ete = u + e + v ,\] and so \( e t e \in U(eRe) \). This gives that \( u + e + U(eRe) \subseteq U(eRe) \), so that \( u + e \in \Delta(eRe) \) implying \( eRe \) is a W$\Delta$U ring, as desired.
\end{proof}

\begin{proposition}\label{matrix}
For any nonzero ring \( R \) and any natural number \( n \geq 2 \), the full matrix ring \( M_n(R) \) is not a W$\Delta$U ring.
\end{proposition}

\begin{proof}
Since \( M_2(R) \) is isomorphic to a corner subring of \( M_n(R) \), it suffices to establish in conjunction with Proposition \ref{corner} that \( M_2(R) \) is not a W$\Delta$U ring. To this goal, considering the matrix unit
$
A = \begin{pmatrix} 0 & 1 \\ 1 & 1 \end{pmatrix},
$
we observe that the sum
\[
A + I = \begin{pmatrix} 0 & 1 \\ 1 & 1 \end{pmatrix} + \begin{pmatrix} 1 & 0 \\ 0 & 1 \end{pmatrix} = \begin{pmatrix} 1 & 1 \\ 1 & 2 \end{pmatrix},
\]
which is a unit and hence cannot be in $\Delta(R)$. Furthermore, the difference
\[
A - I = \begin{pmatrix} 0 & 1 \\ 1 & 1 \end{pmatrix} - \begin{pmatrix} 1 & 0 \\ 0 & 1 \end{pmatrix} = \begin{pmatrix} -1 & 1 \\ 1 & 0 \end{pmatrix},
\]
which is also a unit and thus not in $\Delta(R)$ either. Consequently, we derive that \( M_2(R) \) is not a W$\Delta$U ring.
\end{proof}

A set $\{e_{ij} : 1 \le i, j \le n\}$ of nonzero elements of $R$ is said to be a {\it system of $n^2$ matrix units} provided $e_{ij}e_{st} = \delta_{js}e_{it}$, where $\delta_{jj} = 1$ and $\delta_{js} = 0$ for $j \neq s$. In this case, $e := \sum_{i=1}^{n} e_{ii}$ is an idempotent of $R$ and $eRe \cong M_n(S)$, where $$S = \{r \in eRe : re_{ij} = e_{ij}r,~~\textrm{for all}~~ i, j = 1, 2, . . . , n\}.$$

\medskip

The next affirmation is helpful as well.

\begin{proposition}\label{dedekind}
Every W$\Delta$U ring is Dedekind-finite.
\end{proposition}

\begin{proof}
If we assume on the contrary that $R$ is {\it not} a Dedekind-finite ring, then there exist elements $a, b \in R$ such that $ab = 1$ but $ba \neq 1$. Assuming $e_{ij} = b^i(1-ba)a^j$ and $e =\sum_{i=1}^{n}e_{ii}$, there exists a nonzero ring $S$ such that $eRe \cong M_n(S)$. However, according to Proposition \ref{corner}, $eRe$ is a W$\Delta$U ring, whence $M_n(S)$ must also be a W$\Delta$U ring, thus contradicting Proposition \ref{matrix}.
\end{proof}

\begin{proposition}\label{subring}
Let \(R\) be a W$\Delta$U ring. For an unital subring \(S\) of \(R\), if \(S \cap \Delta(R) \subseteq \Delta(S)\), then \(S\) is a W$\Delta$U ring. In particular, the center of \(R\) is always a W$\Delta$U ring.
\end{proposition}

\begin{proof}
Choosing \(v \in U(S)\) \(\subseteq U(R)\), since \(R\) is W$\Delta$U, we have
\[
v\pm1 \in \Delta(R) \cap S \subseteq \Delta(S).
\] So, \(S\) has to be a W$\Delta$U ring. Now, the rest of our claim follows directly from \cite[Corollary 8]{2}.
\end{proof}

\begin{lemma}\label{reduced}
Let \(R\) be a W$\Delta$U ring. If \(J(R) = \{0\}\) and each nonzero right ideal of \(R\) contains a nonzero idempotent, then \(R\) is reduced.
\end{lemma}

\begin{proof}
Suppose that the contradiction \(R\) is non-reduced holds. Then, there is a nonzero element \(a \in R\) such that \(a^2 = 0\). With \cite[Theorem 2.1]{3} at hand, there is an idempotent \(e \in RaR\) such that \(eRe \cong M_2(T)\) for some non-trivial ring \(T\). Next, Proposition~\ref{corner} informs us that $eRe$ is a W$\Delta$U ring, whence $M_2(T)$ is a W$\Delta$U ring. This, however, contradicts Proposition~\ref{matrix} substantiating our claim.
\end{proof}

\begin{proposition}\label{memeber}
Let \( R \) be a W$\Delta$U ring. Then, the following two equivalencies are true:

(1) \( 3 \in U(R) \iff 4 \in \Delta(R) \).

(2) \( 2 \in U(R) \iff 3 \in \Delta(R) \).

In particular, if \(3 \in \Delta(R)\), then \(Id(R) = \{0, 1\}\).
\end{proposition}

\begin{proof}
(1) Since \( 1 + J(R) \subseteq U(R) \), the implication \( \Leftarrow \) is immediate.
	
For the converse \( \Rightarrow \), assume \( 3 \in U(R) \). Then, because \( R \) is W$\Delta$U, we can write either
\[
3 = 1 + d \quad \text{or} \quad 3 = -1 + d
\]
for some \( d \in \Delta(R) \). In the first case, \( d = 2 \in \Delta(R) \), and, in the second case, \( d = 4 \in \Delta(R) \).
	
Since $\Delta(R)$ is a subring of $R$, it follows that \( 4 \in \Delta(R) \), as pursued.
	
(2) Suppose \( 3 \in \Delta(R) \). Since \( 1 + J(R) \subseteq U(R) \), it follows that \( 2 = -1 + 3 \in U(R) \).
	
Conversely, assume \( 2 \in U(R) \). Then, as \( R \) is W$\Delta$U, we may write either
\[
2 = 1 + d \quad \text{or} \quad 2 = -1 + d
\]
for some \( d \in \Delta(R) \). The first case leads to \( d = 1 \in \Delta(R) \), which is a contradiction. The second case leads to \( d = 3 \in \Delta(R) \), as needed.

Treating now the second part, let $e$ be an idempotent element of $R$. It is easy to verify that $2e-1$ is a unit. Hence, we have either $2e-1=d-1$ or $2e-1=d+1$ for some $d\in \Delta(R)$.

If, foremost, $2e-1=d-1$, then $2e=d$, and so $e=2^{-1}d\in \Delta(R)$ thanks to \cite[Lemma 1(2)]{2}. Next, knowing \cite[Proposition 15]{2}, this allows us to detect that $e=0$.

If, however, $2e-1=d+1$, then $2e=d+2$, and thus $$e=2^{-1}d+1\in 1+\Delta(R)\subseteq U(R),$$ which means that $e=1$.	 \end{proof}

An applicable consequence is the following one.

\begin{corollary}
If \( R \) is a W$\Delta$U ring in which \( 10 \in \Delta(R) \), then \( 4 \in \Delta(R) \).
\end{corollary}

\begin{proof}
Since \( -10 \in \Delta(R) \) and \( 1 + \Delta(R) \subseteq U(R) \), it follows that \( -9 \in U(R) \). Hence, \( 9 = (-3)^2 \in U(R) \) meaning that \( 3 \in U(R) \). Therefore, Proposition~\ref{memeber} works to get that \( 4 \in \Delta(R) \), as claimed.
\end{proof}

\begin{proposition}\label{mino}
A ring \(R\) is $\Delta$U if, and only if, \(R\) is W$\Delta$U and \(2 \in \Delta(R)\).
\end{proposition}

\begin{proof}
The necessity follows immediately from \cite[Proposition~2.4(1)]{7}.

To prove the sufficiency, choose \(u \in R\) to be arbitrary. By assumption, there is an element
\(d \in \Delta(R)\) such that
\[
u = d + 1 \quad \text{or} \quad u = d - 1.
\]
In the latter case, we may rewrite
\[
u = d - 1 = (d - 2) + 1.
\]
Since \(\Delta(R)\) is a subring of $R$, it follows at one that \(d - 2 \in \Delta(R)\). Hence, in either case, \(u\) can be expressed in the required form, as asserted.
\end{proof}

\section{Main Results}

Our initial characterizing statement is the following.

\begin{proposition}\label{division}
Let \( R \) be a ring. Then, the following four conditions are valid:
	
(1) A division ring \( R \) is W$\Delta$U if, and only if, \( R \cong \mathbb{Z}_2 \) or \( \mathbb{Z}_3 \).
	
(2) A local ring \( R \) is W$\Delta$U if, and only if, either \( R/J(R) \cong \mathbb{Z}_2 \) or \(R/J(R) \cong  \mathbb{Z}_3 \).
	
(3) A semi-simple ring $R$ is W$\Delta$U if, and only if, $R \cong \bigoplus_{i=1}^n R_i$, where, for just only one index $i$, $R_i\cong \mathbb{Z}_3 \text{ and } R_j\cong \mathbb{Z}_2$ for any other index $j\neq i$.
	
(4) A semi-local ring $R$ is W$\Delta$U if, and only if, $R/J(R) \cong \bigoplus_{i=1}^m R_i$, where, for just only one index $i$, $R_i\cong \mathbb{Z}_3 \text{ and } R_j\cong \mathbb{Z}_2$ for any other index $j\neq i$.	
\end{proposition}

\begin{proof}
(1) If \( R \) is a division ring, then \( \Delta(R) = 0 \). For any \( x \in R \setminus \{0\} \) we have \( x = \pm1 \), because \( R \) is W$\Delta$U and \( \Delta(R) = 0 \). Thus, \( x^3 = x \) for all \( x \in R \).
	
Note that \( 6 = 2 \cdot 3 = 2^3 - 2 = 0 \). If \( 3 = 0 \), then exploiting \cite[Ex.\ 12.11]{lam}, \( R \) is a subdirect product of copies of \( \mathbb{Z}_3 \), whence \( R \cong \mathbb{Z}_3 \).
	
If \( 2 = 0 \), then \( (1 + y)^3 = 1 + y \) guarantees that \( y = 0 \) or \( y = 1 \), that is, \( R \cong \mathbb{Z}_2 \). The converse is pretty straightforward.
	
(2) Assume one of the possibilities \( R/J(R) \cong \mathbb{Z}_2 \) or \(R/J(R) \cong  \mathbb{Z}_3 \). We know that \(R/J(R)\) is a division ring, and so \(R/J(R)\) is 2-$\Delta$U viewing (1). Thus, \(R\) is W$\Delta$U utilizing Corollary~\ref{2}.
	
Conversely, letting \(R\) be W$\Delta$U, we obviously check with the help of (1) that \(R/J(R) \cong \mathbb{Z}_2\) or \(R/J(R) \cong \mathbb{Z}_3\).
	
(3) If $R$ is a semi-simple ring, then consulting with the classical Wedderburn-Artin theorem, we obtain an isomorphism $$R \cong \bigoplus_{i=1}^n M_{n_i}(D_i),$$ where each $D_i$ is a division ring. Since $R$ is W$\Delta$U, Proposition~\ref{productt} is a guarantor that one of the rings $M_{n_i}(D_i)$ must also be W$\Delta$U and the other is necessarily $\Delta$U. Combining Proposition~\ref{matrix} and \cite[Theorem 2.5]{7}, it follows that $n_i=1$ for each $i$, so that one of $D_i$ itself is W$\Delta$U and the others are just $\Delta$U. Adapting (1) and \cite[Proposition 2.4(2)]{7}, we come to the fact that one of $D_i$ is isomorphic to $\mathbb{Z}_3$ and the others are isomorphic to $\mathbb{Z}_2$. The converse claim is obvious looking at Proposition \ref{productt} and (1).
	
(4) This is quite evident from points (3) and (1).
\end{proof}

\begin{proposition}\label{7}
Weakly $\Delta$-clean rings are W$\Delta$U rings.
\end{proposition}

\begin{proof}
Choosing \( u \in U(R) \), we can write \( u = d \pm e \) for some \( d \in \Delta(R) \) and \( e \in Id(R) \). Since \( U(R) + \Delta(R) = U(R) \), it follows that \( u - d = \pm e \in U(R) \). Therefore, \( e = 1 \) and so \( u = d \pm 1 \), as required.	
\end{proof}

\begin{proposition}\label{8}
Every weakly $\Delta$-clean ring is clean.
\end{proposition}

\begin{proof}
Let \( r \in R \) be arbitrary. Then, either \( r - 1 = d + e \) or \( r - 1 = d - e \), where \( d \in \Delta(R) \) and \( e \in Id(R) \). In the first case, we write
$
r = 1 + d + e.
$
In the second case, we write
\[
r = 1 + d - e = (1 - 2e + d) + e.
\]
It is plainly verified that
$
(1 - 2e)^2 = 1,
$
and thus \( 1 - 2e \) is a unit. Using \cite[Lemma 2.1]{Dj}, both \( 1 - 2e + d \) and \( 1 + d \) are units, as required.
\end{proof}

\begin{proposition}\label{6}
The following two conditions are equivalent for a ring \( R \):

(1) \( R \) is a W\( \Delta U \) ring.

(2) All clean elements of \( R \) are weakly \( \Delta \)-clean.
\end{proposition}

\begin{proof}
(1) $\Rightarrow$ (2). Assume that \( R \) is a W\( \Delta U \)-ring. Let \( r \in R \) be an arbitrary clean element with clean decomposition \( r = e + u \). Since \( R \) is a W\( \Delta U \)-ring, we perceive \( u = \pm 1 + d \) for some \( d \in \Delta(R) \).

First, consider the case \( u = 1 + d \). Note that \( 1 - 2e \in U(R) = 1 + \Delta(R) \), and so \( 2e \in \Delta(R) \). It follows that \( 2e + d \in \Delta(R) \), and hence
\[
r = e + 1 + d = (1 - e) + (2e + d).
\]

Now, consider the second case \( u = -1 + d \). Then, we may write
\[
r = e - 1 + d = -(1 - e) + d.
\]

Thus, in both cases, \( r \) admits a weakly \( \Delta \)-clean decomposition. Therefore, condition (2) holds.

(2) $\Rightarrow$ (1). Let \( u \in U(R) \). Then, \( u \) is a clean element, and hence by condition (2) the element \( u \) is weakly \( \Delta \)-clean. So, we can write \( u = \pm e + d \) for some idempotent \( e \in R \) and \( d \in \Delta(R) \).

Next, multiplying both sides by \( u^{-1} \), we come to the equality
\[
1 = \pm eu^{-1} + du^{-1},
\]
leading to
\[
\pm eu^{-1} = 1 - du^{-1}.
\]
Since the utilization of \cite[Lemma 1]{2} shows that \( du^{-1} \in \Delta(R) \), it follows that \[ \pm eu^{-1} \in 1 + \Delta(R) \subseteq U(R) .\] Hence, \( e = 1 \) and, therefore,
\[
u = \pm1 + d \in \pm1 + \Delta(R),
\]
showing that \( U(R) = \pm1 + \Delta(R) \).
\end{proof}

Our first major result is the following.

\begin{theorem}\label{9}
For a ring \( R \), the following four conditions are equivalent:

(1) \( R \) is a clean W\( \Delta U \) ring.

(2) For each \( a \in R \), we have \( a \pm a^2 \in \Delta(R) \) and \( a \pm e \in \Delta(R) \) for some idempotent \( e \in R \).

(3) \( R \) is a weakly \( \Delta \)-clean W\( \Delta U \) ring.

(4) \( R \) is a weakly \( \Delta \)-clean ring.
\end{theorem}

\begin{proof}
(1) \( \Rightarrow \) (2). Assume that \( R \) is a clean W\( \Delta U \) ring. Then, Proposition~\ref{6} tells us that, for any \( a \in R \), there is an idempotent \( e \in R \) such that \( a \pm e \in \Delta(R) \). Moreover,  Proposition~\ref{6} reaches us that there is a weakly \( \Delta \)-clean decomposition \( a = \pm e + d \) with \( d \in \Delta(R) \).

First, considering the case \( a = e + d \), we may write
\[
a - a^2 = (e + d) - (e + d)^2 = d - d^2 - (ed + de).
\]
Since \( d - d^2 \in \Delta(R) \) and \( 2ed \in \Delta(R) \), it suffices to prove that \( ed + de \in \Delta(R) \). Indeed, one observes that
\[
(ed(1 - e))^2 = 0 = ((1 - e)de)^2.
\]
It thus follows that
\[
ed - ede = ed(1 - e) \in \Delta(R), \quad de - ede = (1 - e)de \in \Delta(R),
\]
meaning that
\[
de - ed \in \Delta(R), \quad \text{and hence} \quad ed + de = 2ed + (de - ed) \in \Delta(R).
\]
So, \( a - a^2 \in \Delta(R) \).

Now, consider \( a = -e + d \). Then, we can write
\[
a + a^2 = (d + d^2) - (ed + de).
\]
As in the previous case, we again conclude that \( ed + de \in \Delta(R) \), and hence \( a + a^2 \in \Delta(R) \), as asked for.

(2) \( \Rightarrow \) (3). Evidently, \( R \) is a weakly \( \Delta \)-clean ring. Let \( u \in U(R) \). Then, by condition (2), it follows that \( u \pm u^2 \in \Delta(R) \). Consequently, we deduce \( 1 \pm u \in \Delta(R) \) employing \cite[Lemma 1(2)]{2}. So, \( u = \pm1 + d \) for some \( d \in \Delta(R) \), whence \( R \) is a W\( \Delta U \) ring.

(3) \( \Rightarrow \) (4). This is rather simple.

(4) \( \Rightarrow \) (1). This relation follows at once invoking Propositions~\ref{7} and \ref{8}.
\end{proof}

Our second principal result is the following.

\begin{theorem}\label{main}
Suppose \( R \) is a ring. Then, the following four conditions are equivalent:
	
(1) \( R \) is exchange W$\Delta$U.
	
(2) \( R \) is clean W$\Delta$U.
	
(3) \( R \) is weakly $\Delta$-clean.
	
(4) All idempotents of \( R \) lift modulo \( J(R) \), and either
\[
	R / J(R) \cong B, \quad \text{or} \quad R / J(R) \cong \mathbb{Z}_3, \quad \text{or} \quad R / J(R) \cong B \times \mathbb{Z}_3,
\]
where \( B \) is a Boolean ring.
\end{theorem}

\begin{proof}
The equivalence \( (2) \Leftrightarrow (3) \) follows at once from Theorem~\ref{9}.

(2) \( \Rightarrow \) (1). This is automatic.

(1) \( \Rightarrow \) (2). We observe that \( R/J(R) \) is an exchange ring and all idempotents lift modulo \( J(R) \) taking into account \cite[Proposition 1.5]{8}. Furthermore, Corollary~\ref{2} allows us to deduce that \( R/J(R) \) is a W\( \Delta U \) ring. Since \( J(R/J(R)) = 0 \), it follows from Lemma~\ref{reduced} that \( R/J(R) \) is reduced and so abelian. Consequently, bearing in mind \cite[Proposition 1.8(2)]{8}, \( R/J(R) \) is an abelian exchange ring, and therefore clean. Given that idempotents lift modulo \( J(R) \), we conclude that \( R \) is clean.

(1) \( \Rightarrow \) (4). Using a similar argumentation to that of above, we infer that \( R/J(R) \) is abelian and weakly clean. Hence, \cite[Proposition 14]{9} employs to obtain that \( R/J(R) \cong M_n(D) \), where \( 1 \leq n \leq 2 \) and \( D \) is a division ring. Consequently, \cite[Theorem 11]{2} applies to extract that \( \Delta(R) = J(R) \). Thus, \( R \) is an exchange WUJ ring. The wanted result then follows from \cite[Theorem 2.2]{wuj}.

(4) \( \Rightarrow \) (1). Applying \cite[Theorem~2.2]{wuj}, it follows that $R$ is an exchange $WJU$ ring and, consequently, $R$ is also an exchange $W\Delta U$-ring.
\end{proof}

In addition to the presentation given above, a ring $R$ is said to be {\it weakly Boolean} if, for any $a\in R$, either $a$ or $-a$ is an idempotent. Likewise, a ring $R$ is called {\it semi-weakly Boolean} if $R/J(R)$ weakly Boolean and its idempotents lift modulo $J(R)$.

\medskip

Our third chief result states thus.

\begin{theorem}\label{3.16}
Let $R$ be a ring. Then, the following three statements are equivalent:

(1) $R$ is a semi-regular $W\Delta U$-ring.

(2) $R$ is an exchange $W\Delta U$-ring.

(3) $R$ is a semi-weakly Boolean ring.
\end{theorem}

\begin{proof}
$(1) \Rightarrow (2)$. This implication is immediate, because each semi-regular ring is exchange.
	
$(2) \Rightarrow (3)$. Viewing \cite[Corollary~2.4]{8}, the quotient-ring $R/J(R)$ is exchange and all
idempotents lift modulo $J(R)$. Moreover, Corollary~\ref{2} teaches us that $R/J(R)$ is a $W\Delta U$-ring. Hence, without loss of generality, we may assume that $J(R)=\{0\}$.
	
Furthermore, since $R$ is an exchange ring, every nonzero one-sided ideal of $R$ contains a nonzero idempotent. So, Lemma~\ref{reduced} can be applied to derive that $R$ is reduced, and hence abelian. Therefore, $R$ is abelian clean. However, with the knowledge of \cite[Proposition~14]{9}, we may write
$
R/J(R)\cong M_{n}(D),
$
where $1\leq n\leq 2$ and $D$ is a division ring. Consequently, \cite[Theorem~11]{2} yields that $\Delta(R)=J(R)$, and therefore $\Delta(R)=Nil(R)=J(R)=\{0\}$.

Now, as $R$ is a $WUU$ exchange ring, it follows from \cite[Corollary~2.15]{wuu} that $R$ is strongly weakly nil-clean. Thus, the conclusion follows from \cite[Theorem~3.2]{ch}.
	
$(3) \Rightarrow (1)$. Since the factor-ring $R/J(R)$ is weakly Boolean, it is straightforwardly that $R/J(R)$ is regular and $W\Delta U$. Hence, by Corollary~\ref{2}, $R$ is a semi-regular $W\Delta U$-ring, as formulated.
\end{proof}

Two traditional consequences are like these.

\begin{corollary}\label{3.17}
Let $R$ be a $W\Delta U$-ring. Then, the following three statements are equivalent:

(1) $R$ is a semi-regular ring.

(2) $R$ is an exchange ring.

(3) $R$ is a clean ring.
\end{corollary}

\begin{proof}
$(1) \Leftrightarrow (2)$. It follows directly from Theorem~\ref{3.16}.
	
$(3) \Rightarrow(2) $. This implication is well-known and follows immediately, so the details are skipped.
	
$(2) \Rightarrow (3)$. Assume that $R$ is an exchange $W\Delta U$-ring. Thanks to Lemma~\ref{reduced}, $R$ is reduced whence abelian. Therefore, $R$ is an abelian exchange ring, i.e., $R$ is an abelian clean. 
\end{proof}

\begin{corollary}\label{3.19}
Let $R$ be an exchange ring. Then, the following two statements are equivalent:

(1) $R$ is a $W\Delta U$-ring.

(2) $R$ is a $WUJ$-ring.
\end{corollary}

\begin{proof}
$(2) \Rightarrow (1)$. This relation is immediate, since the inclusion $J(R)\subseteq \Delta(R)$ holds for every ring.
	
$(1) \Rightarrow (2)$. Assume that $R$ is an exchange $W\Delta U$-ring. Arguing as in the proof of Theorem~\ref{3.16}, we obtain
$
J(R)=\Delta(R)=\{0\}.
$
Consequently, $R$ is a $WUJ$ ring, as stated.
\end{proof}

Our next characterization criterion is the following.

\begin{theorem}\label{fin2}
A ring \( R \) is weakly clean W\(\Delta U\) with \( 2 \in U(R) \) if, and only if, \( R / J(R) \cong \mathbb{Z}_3 \).
\end{theorem}

\begin{proof}
Since the sufficiency has already been established in Theorem \ref{main}, we now turn our attention to the necessity. We manage to demonstrate that \( R \) is indecomposable; that is, the only idempotents in \( R \) are \( \{0, 1\} \).

To this intention, consider an arbitrary idempotent \( e \in R \). Note that the element \( 2e - 1 \) is a unit in \( R \), and hence it can be written as either \( 2e - 1 = 1 + d \) or \( 2e - 1 = -1 + d \), where \( d \in \Delta(R) \). Consequently, we obtain \( 2e = 2 + d \) or \( 2e = d \), yielding that \( e = 1 + \frac{d}{2} \) or \( e = \frac{d}{2} \). In both cases, \( e \in U(R) \) or \( e \in \Delta(R) \), asserting that \( e = 1 \) or \( e = 0 \).

That is why, there are precisely six possibilities for any element \( r \in R \) expressed in terms of \( d \in \Delta(R) \) as follows:

\medskip

\begin{itemize}
	\item \( r = 1 + d + 1 = 2 + d \in U(R) + \Delta(R) \subseteq U(R) \),
	\item \( r = 1 + d - 1 = d \in \Delta(R) \),
	\item \( r = 1 + d \in 1 + \Delta(R) \subseteq U(R) \),
	\item \( r = -1 + d - 1 = -2 + d \in U(R) \),
	\item \( r = -1 + d + 1 = d \in \Delta(R) \),
	\item \( r = -1 + d \in U(R) \).
\end{itemize}

\medskip

Hence, each element of \( R \) is either a unit or lies in \( \Delta(R) \). This enables us that \( R \) is a local ring, and so \( R/J(R) \) is a division ring.

To verify this, let \( x \in R \) be such that \( x \not\in J(R) \). Then, there exists \( y \in R \) such that \( 1 - xy \not\in U(R) \) enabling us that \( 1 - xy \in \Delta(R) \), and thus \( xy \in 1 + \Delta(R) \subseteq U(R) \). Also, for some \( h \in R \), the element \( hx \) is also a unit, guaranteeing that \( x \) itself must be a unit too.

Moreover, since \( R/J(R) \) is a W\(\Delta\)U-ring in agreement with Corollary \ref{2}, as well as it is simultaneously a division ring, it follows from Proposition \ref{division} that \( R/J(R) \cong \mathbb{Z}_3 \), as needed.
\end{proof}

\section{Some Extensions of $W\Delta U$-Rings}

\begin{proposition}
Let \( R \) be a ring. Then, the following four conditions are equivalent:

(1) \( R \) is $\Delta U$.

(2) \( T_n(R) \) is $\Delta U$ for all \( n \in \mathbb{N} \).

(3) \( T_n(R) \) is $\Delta U$ for some \( n \in \mathbb{N} \).

(4) \( T_n(R) \) is $W\Delta U$ for some \( n \geq 2 \).
\end{proposition}

\begin{proof}
(1) \( \Leftrightarrow \) (2) \( \Leftrightarrow \) (3). These relationships are already proved in \cite[Corollary 2.9]{7}.

(1) \( \Rightarrow \) (4). This is pretty trivial.

(4) \( \Rightarrow \) (1). Choose $u\in U(R)$ and set
$$A:=\begin{pmatrix}
	u & & & & \ast \\
	& -u & & & \\
	& & 1 & & \\
	& & & \ddots & \\
	0 & & & & 1
\end{pmatrix}\in U(T_n(R)).$$ By hypothesis, we can find
$\begin{pmatrix}
	d_{1} & & & \ast \\
	& d_{2} & & \\
	& & \ddots & \\
	0 & & & d_{n}
\end{pmatrix}$ in $\Delta(T_n(R))$ \\
such that
$$A=\pm \begin{pmatrix}
	1 & & & \ast \\
	& 1 & & \\
	& & \ddots & \\
	0 & & & 1
\end{pmatrix}+\begin{pmatrix}
	d_{1} & & & \ast \\
	& d_{2} & & \\
	& & \ddots & \\
	0 & & & d_{n}
\end{pmatrix}.$$
By a plain check, it now follows that $u=1+d_1$ or $u=1-d_2$. Working with \cite[Corollary 9(1)]{2}, one finds that both $d_1$ and $d_2$ are in $\Delta(R)$, thus proving point (1), as requested.
\end{proof}

\begin{proposition}\label{3.9}
The ring $R[[x;\alpha]]$ is $W\Delta U$ if, and only if, the ring $R$ is $W\Delta U$.
\end{proposition}

\begin{proof}
Put $I := R[[x;\alpha]]x$. Then, one inspects that $I$ is a proper ideal of $R[[x;\alpha]]$. But, it is well known that
\[
	J(R[[x;\alpha]]) = J(R) + I,
\]
and hence $I \subseteq J(R[[x;\alpha]])$. Moreover, since
	$
	R[[x;\alpha]]/I \cong R,
	$
the statement follows from Proposition~\ref{1} immediately.
\end{proof}

As a direct consequence, we have the following.

\begin{corollary}
The ring $R[[x]]$ is $W\Delta$U if, and only if, so is $R$.
\end{corollary}

Let \(R\) be a ring, and \(\alpha\) an endomorphism of \(R\). We say the ring \(R\) is \text{\(\alpha\)-compatible} if, for all \(a, b \in R\),
\[
ab = 0 \quad \Longleftrightarrow \quad a\alpha(b) = 0.
\]

\medskip

A few non-elementary results are the following ones.

\begin{proposition}\label{2.10}
Let \(R\) be a weakly \(2\)-primal, \(\alpha\)-compatible ring. Then, the following equality is fulfilled:
\[
	\Delta(R[x; \alpha]) = \Delta(R) + L(R)[x; \alpha]x.
\]
\end{proposition}

\begin{proof}
Take \(f := \sum_{i=0}^n a_i x^i \in \Delta(R[x; \alpha])\). For any \(u \in U(R)\), we have \(1 - uf \in U(R[x; \alpha])\). By \cite[Corollary~2.14]{15}, it follows that \(a_0 \in \Delta(R)\) and \(a_i \in \operatorname{Nil}(R) = L(R)\) for all \(1 \le i \le n\). Hence, \(\Delta(R[x; \alpha]) \subseteq \Delta(R) + L(R)[x; \alpha]x\).
	
Conversely, suppose \(f := \sum_{i=0}^n a_i x^i \in \Delta(R) + L(R)[x; \alpha]x\) and let \(U = \sum_{i=0}^n u_i x^i \in U(R[x; \alpha])\). Again by \cite[Corollary~2.14]{15}, we know \(u_0 \in U(R)\) and \(u_i \in \operatorname{Nil}(R) = L(R)\) for \(1 \le i \le n\). Because \(\operatorname{Nil}(R) = L(R)\) is an ideal, a routine computation shows that \(1 - uf \in U(R) + L(R)[x; \alpha]x\), and so \cite[Corollary~2.14]{15} can be employed once more to yield \(1 - uf \in U(R[x; \alpha])\), which means \(f \in \Delta(R[x; \alpha])\), as required.
\end{proof}

Our main achievement is this section is the following.

\begin{theorem}\label{3.29}
Let $R$ be a weakly $2$-primal ring and let $\alpha$ be an endomorphism of $R$ such that $R$ is $\alpha$-compatible. Then, the following two statements are equivalent:

(1) $R[x;\alpha]$ is a $W\Delta U$-ring.

(2) $R$ is a $W\Delta U$-ring.
\end{theorem}

\begin{proof}
$(1) \Rightarrow (2)$. Given $u\in U(R)$, we have $u\in U(R[x;\alpha])$ and, by
assumption,
\[
u\in \pm1+\Delta(R[x;\alpha])
=\pm1+\Delta(R)+L(R)[x; \alpha]x.
\]
Therefore, $u\in \pm1+\Delta(R)$, and thus $R$ is a $W\Delta U$-ring.

$(2) \Rightarrow(1) $. Take an arbitrary unit \(u(x) = \sum_{i=0}^n a_i x^i \in U(R[x;\alpha])\). Inspired from \cite[Corollary~2.14]{15}, we infer \(a_0 \in U(R)\) and \(a_i \in \operatorname{Nil}(R)\) for all \(i \ge 1\). Since \(R\) is a \(W\Delta U\)-ring, there is \(d \in \Delta(R)\) such that \(a_0 = \pm 1 + d\).

A direct computation now gives
\[
u(x) = a_0 + \sum_{i=1}^n a_i x^i = (\pm 1 + d) + \sum_{i=1}^n a_i x^i,
\]
so that
\[
u(x) \in \pm 1 + \Delta(R) + \operatorname{Nil}(R)[x;\alpha]x.
\]

However, because \(R\) is weakly \(2\)-primal, \(\operatorname{Nil}(R) = L(R)\). Hence,
\[
u(x) \in \pm 1 + \Delta(R) + L(R)[x;\alpha]x.
\]

Furthermore, utilizing Proposition~\ref{2.10}, we obtain \(\Delta(R) + L(R)[x;\alpha]x = \Delta(R[x;\alpha])\). Consequently,
\[
u(x) \in \pm 1 + \Delta(R[x;\alpha]),
\]
which equals to this that \(R[x;\alpha]\) is a \(W\Delta U\) ring.
\end{proof}

\begin{corollary}\label{3.30}
Let \( R \) be a weakly \(2\)-primal ring. Then, the following two assertions are equivalent:

(1) \( R[x] \) is a $W\Delta$U ring.

(2) \( R \) is a $W\Delta$U ring.
\end{corollary}

\begin{proposition}\label{3}
Let $R$ be a ring, and let
\begin{center}
		${\rm S}_{n}(R)=\left\lbrace (a_{ij})\in {\rm T}_{n}(R)\, | \, a_{11}=a_{22}=\cdots=a_{nn}\right\rbrace.$
\end{center}
Then, the following two points are equivalent:
	
(1) $R$ is a $W\Delta$U ring.
	
(2) ${\rm S}_{n}(R)$ is a $W\Delta$U ring.
\end{proposition}

\begin{proof}
Put \(I := \{ (a_{ij}) \in S_n(R) \mid a_{11} = 0 \}\), whence \(I\subseteq J(S_n(R))\) and, therefore, \(S_n(R)/I \cong R\). Hence, Proposition~\ref{1} is in use to obtain that the desired conclusion follows.	
\end{proof}

Let \(R\) be a ring, and let \(M\) be an \(R\)-bi-module. The \text{trivial extension} of \(R\) by \(M\) is defined as the set
\[
T(R, M) = \{ (r, m) \mid r \in R,\ m \in M \},
\]
with componentwise addition and multiplication given by
\[
(r, m)(s, n) = (rs,\ rn + ms).
\]

This ring is naturally isomorphic to the subring of the \(2 \times 2\) formal matrix ring \(\begin{pmatrix} R & M \\ 0 & R \end{pmatrix}\) consisting of matrices of the form
$\begin{pmatrix} r & m \\ 0 & r \end{pmatrix}$, where $r \in R$ and $m \in M$.

\medskip

In the special case when \(M = R\), we have \(T(R, R) \cong R[x]/(x^2)\).

\medskip

The group of units of \(T(R, M)\) is given by
\[
U(T(R, M)) = T(U(R), M).
\]
Moreover, regarding \cite{7}, the \(\Delta\)-set of the trivial extension satisfies the equality.
\[
\Delta(T(R, M)) = T(\Delta(R), M).
\]

\begin{proposition}\label{trivial}
The trivial extension $T(R, M)$ is a \(W\Delta U\) ring if, and only if, $R$ is a \(W\Delta U\) ring.
\end{proposition}

\begin{proof}
The necessary argument follows along the same lines as the proof of Proposition~\ref{3}.
\end{proof}

Consider now the following set of triangular matrices:
\[
T(R,n) := \left\{
\begin{pmatrix}
	a_1 & a_2 & a_3 & \dots & a_n \\
	0 & a_1 & a_2 & \dots & a_{n-1} \\
	0 & 0 & a_1 & \dots & a_{n-2} \\
	\vdots & \vdots & \vdots & \ddots & \vdots \\
	0 & 0 & 0 & \dots & a_1
\end{pmatrix}
\;\Bigg|\; a_i \in R \right\}, \quad n \ge 2.
\]

It is readily verified that \(T(R,n)\) forms a subring of $T_n(R)$ under ordinary matrix addition and multiplication. By identifying an element of \(T(R,n)\) with the \(n\)-tuple \((a_1, a_2, \dots, a_n)\), one may describe the ring structure as follows: operation addition is componentwise, and operation multiplication is given by
\[
(a_1, a_2, \dots, a_n)(b_1, b_2, \dots, b_n) = (a_1b_1,\ a_1b_2 + a_2b_1,\ \dots,a_1b_n+a_2b_{n-1}+\dots+a_nb_1),
\]
for all \(a_i, b_j \in R\). In fact, \(T(R,n)\) is isomorphic to the quotient ring \(R[x]/(x^n)\).

\medskip

One knows that Wang introduced in \cite{wang} the matrix ring $S_{n,m}(R)$. Supposing $R$ is a ring, then the matrix ring $S_{n,m}(R)$ can be represented as

$$\left\{\begin{pmatrix}
	a & b_1 & \cdots & b_{n-1} & c_{1n} & \cdots & c_{1 n+m-1}\\
	\vdots  & \ddots & \ddots & \vdots & \vdots & \ddots & \vdots \\
	0 & \cdots & a & b_1 & c_{n-1,n} & \cdots & c_{n-1,n+m-1} \\
	0 & \cdots & 0 & a & d_1 & \cdots & d_{m-1} \\
	\vdots  & \ddots & \ddots & \vdots & \vdots & \ddots & \vdots \\
	0 & \cdots & 0 & 0  & \cdots & a & d_1 \\
	0 & \cdots & 0 & 0  & \cdots & 0 & a
\end{pmatrix}\in T_{n+m-1}(R) : a, b_i, d_j,c_{i,j} \in R \right\}.$$

Likewise, let 

$$T_{n,m}(R)=\left\{ \left(\begin{array}{@{}c|c@{}}
	\begin{matrix}
		a & b_1 & b_2 & \cdots & b_{n-1} \\
		0 & a & b_1 & \cdots & b_{n-2} \\
		0 & 0 & a & \cdots & b_{n-3} \\
		\vdots & \vdots & \vdots & \ddots & \vdots \\
		0 & 0 & 0 & \cdots & a
	\end{matrix}
	& \bigzero \\
	\hline
	\bigzero &
	\begin{matrix}
		a & c_1 & c_2 & \cdots & c_{m-1} \\
		0 & a & c_1 & \cdots & c_{m-2} \\
		0 & 0 & a & \cdots & c_{m-3} \\
		\vdots & \vdots & \vdots & \ddots & \vdots \\
		0 & 0 & 0 & \cdots & a
	\end{matrix}
\end{array}\right)\in T_{n+m}(R) : a, b_i,c_j \in R \right\},$$

\medskip

\noindent and let

\medskip

$$U_{n}(R)=\left\{ \begin{pmatrix}
	a & b_1 & b_2 & b_3 & b_4 & \cdots & b_{n-1} \\
	0 & a & c_1 & c_2 & c_3 & \cdots & c_{n-2} \\
	0 & 0 & a & b_1 & b_2 & \cdots & b_{n-3} \\
	0 & 0 & 0 & a & c_1 & \cdots & c_{n-4} \\
	\vdots & \vdots & \vdots & \vdots &  &  & \vdots \\
	0 &0 & 0 & 0 & 0 & \cdots & a
\end{pmatrix}\in T_{n}(R) :  a, b_i, c_j \in R \right\}.$$	

So, we routinely obtain the following.

\begin{proposition}
Let $R$ be a ring. Then, the following four items are equivalent:

(1) $R$ is a \(W\Delta U\) ring.

(2) $T(R,n)$ is a \(W\Delta U\) ring.

(2) $S_{n,m}(R)$ is a \(W\Delta U\) ring.

(3) $T_{n,m}(R)$ is a \(W\Delta U\) ring.

(4) $U_{n}(R)$ is a \(W\Delta U\) ring.
\end{proposition}

\begin{proof}
The proof is, notably, analogous to that of Proposition~\ref{3}, so we omit the details.
\end{proof}

We are now concerned with group rings. Let \(R\) be an arbitrary ring and let \(G\) be an arbitrary group. As is customary, \(RG\) denotes the group ring of \(G\) over \(R\). The kernel of the canonical augmentation map \(\varepsilon: RG \to R\), given by \(\varepsilon\big(\sum_{g\in G} a_g g\big) = \sum_{g\in G} a_g\), is denoted by \(\varepsilon(RG)\) and is referred to as the augmentation ideal of \(RG\).

\medskip

Recall that a group \(G\) is a {\it \(p\)-group} provided that every element of \(G\) has order which is a power of a fixed prime \(p\). Furthermore, \(G\) is said to be {\it locally finite} if each finitely generated subgroup of \(G\) is finite.

\medskip

We call a subring $S$ of $R$ a {\it good subring} whenever $U(R) \cap S = U(S)$.

\begin{lemma}\label{goodsubring}
Let $R$ be a \(W\Delta U\) ring and \(S\) a good subring of \(R\). Then, \(S\) is a \(W\Delta U\) ring. In particular, this always applies to the center \(C = C(R)\) of \(R\).
\end{lemma}

\begin{proof}
The argument follows deducing as in the proof of \cite[Proposition 2.4(7)]{7}, so we drop off details.
\end{proof}

\begin{lemma}\label{4.14} \cite[Lemma $2$]{26}.
Let $p$ be a prime with $p\in J(R)$. If $G$ is a locally finite $p$-group, then $\varepsilon(RG) \subseteq J(RG)$.
\end{lemma}

The following claim is our pivotal tool for proving the central results on \(W\Delta U\) group rings listed below.

\begin{proposition}
(1) If $RG$ is a \(W\Delta U\) ring, then $R$ is also a \(W\Delta U\) ring.

(2) If \( R \) is a \(W\Delta U\) ring and \( G \) is a locally finite \( p \)-group, where \( p \) is a prime number such that \( p \in J(R) \), then \( RG \) is a \(W\Delta U\) ring.
\end{proposition}

\begin{proof} We present two different ideas of a confirmation of the statement like these:

\medskip

(1) \textbf{Method 1:} Let \(u \in U(R)\), whence \(u \in U(RG)\) as well. Hence, \(u = \pm1 + r\) for some \(r \in \Delta(RG)\). Since \(r = \pm1 - u \in R\), it remains to check that \(r \in \Delta(R)\). This is easy to check: in fact, for any \(v \in U(R) \subseteq U(RG)\), we have \(v - r \in U(RG) \cap R \subseteq U(R)\). Thus, \(r \in \Delta(R)\).

\medskip
	
\textbf{Method 2:} Because \(R\) is a good subring of \(RG\), the claim follows directly from Lemma~\ref{goodsubring}.

\medskip
	
(2) Observe that Lemma~\ref{4.14} implies \(\varepsilon(RG) \subseteq J(RG)\). Moreover, since one has that \(RG / \varepsilon(RG) \cong R\), an application of Proposition~\ref{1} yields that \(RG\) is a \(W\Delta U\) ring.
\end{proof}

\begin{proposition}
Let \(R\) be a ring, \(G\) a group, and \(H\) a subgroup of \(G\). If \(RG\) is a \(W\Delta U\) ring, then \(RH\) is also a \(W\Delta U\) ring.
\end{proposition}

\begin{proof}
Since \(RH\) is a good subring of \(RG\), the conclusion follows immediately from Lemma~\ref{goodsubring}.
\end{proof}

\begin{lemma}\label{torsion}
If \(RG\) is a \(W\Delta U\) ring, then \(G\) must be a torsion group.
\end{lemma}

\begin{proof}
Suppose on the contrary that some element \(g \in G\) has infinite order. Since \(RG\) is a \(W\Delta U\) ring, we have either \(1 - g \in \Delta(RG)\) or \(1 + g \in \Delta(RG)\). So, \cite[Lemma 2.1(3)]{7} illustrates that either \(g(1 - g)=g-g^2 \in \Delta(RG)\) or \(g(1 +g)=g+g^2 \in \Delta(RG)\), so it must be that \(1 + g - g^2 \in U(RG)\) or \(1 + g + g^2 \in U(RG)\). We focus on the case where \(1 + g - g^2 \in U(RG)\). Then, there exist two integers \(n < m\) and coefficients \(a_i \in R\) with \(a_n \neq 0 \neq a_m\) such that
\[
(1 + g - g^2) \sum_{i=n}^m a_i g^i = 1.
\]
Expanding the product by the Newton's binomial formula leads us directly to a contradiction. A similar argument applied to \(1 + g + g^2 \in U(RG)\) yields the same contradiction. Hence, every element of \(G\) must have finite order, that means \(G\) is a torsion group, as claimed.
\end{proof}

\begin{proposition}\label{torsion2}
If \(RG\) is a \(W\Delta U\) ring and \(2 \in \Delta(RG)\), then \(G\) must be a \(2\)-group.	
\end{proposition}

\begin{proof}
Since \(2 \in \Delta(RG)\), it follows from Proposition~\ref{mino} that \(RG\) is, in fact, a \(\Delta U\) ring.

\medskip

\textbf{Claim:} For any \(g \in G\) and any positive integer \(k\), the sum \(\sum_{i=0}^{2k} g^i\) lies in \(U(RG)\).

\medskip

We need to demonstrate the method only for \(k = 1\) and \(k = 2\) as the general case follows by a standard inductive argument.
	
\medskip

\textbf{Case \(k = 1\):} Since \(2 \in \Delta(RG)\) and \(g \in U(RG)\), Lemma 2.1(3) of \cite{7} yields \(2g \in \Delta(RG)\). Since \(1 - g \in \Delta(RG)\) and \(2g \in \Delta(RG)\), Lemma 2.1(1) from \cite{7} yields \(1 + g \in \Delta(RG)\). Hence, \(g(1+g) = g + g^2 \in \Delta(RG)\), and therefore \(1 + g + g^2 \in U(RG)\).

\medskip
		
\textbf{Case \(k = 2\):} As before, \(g, g^2 \in U(RG)\), while we already have \(g + g^2 \in \Delta(RG)\).
Multiplying by \(g^2\) gives \(g^3 + g^4 \in \Delta(RG)\). Consequently, \(g + g^2 + g^3 + g^4 \in \Delta(RG)\), and thus
\[
1 + g + g^2 + g^3 + g^4 \in U(RG).
\]
	
Proceeding inductively, one establishes that \(\sum_{i=0}^{2k} g^i \in U(RG)\) for every \(k \ge 1\).

Now, we assume that some element \(g \in G\) has order \(p\) not dividing \(2\). Then, \(p\) is odd, so we can write \(p-1 = 2k\) for a suitable \(k\). By the above Claim, one knows that \(\sum_{i=0}^{2k} g^i \in U(RG)\). However,
\[
(1 - g) \sum_{i=0}^{2k} g^i = 1 - g^{p} = 0,
\]
and multiplying by the inverse of \(\sum_{i=0}^{2k} g^i\) perceives that \(1 - g = 0\), i.e., \(g = 1\) -- a contradiction. Finally, each element of \(G\) has order which is a power of \(2\); in other words, \(G\) is a \(2\)-group after all, as asserted.
\end{proof}

The following necessary and sufficient condition is worthy of recording.

\begin{proposition}
Let \(R\) be an Artinian ring. Then, \(RG\) is a \(W\Delta U\) ring if, and only if, \((R/J(R))G\) is a \(W\Delta U\) ring.
\end{proposition}

\begin{proof}
Since \(R\) is Artinian, we may apply \cite[Proposition 3]{con} to obtain \(J(R)G \subseteq J(RG)\). Moreover, the natural isomorphism \((R/J(R))G \cong RG/J(R)G\) is always fulfilled. The desired equivalence now follows at once from Corollary~\ref{2}.
\end{proof}

We finish our work with the following.

\begin{lemma}
Let \(RG\) be a \(W\Delta U\) ring with \(3 \in \Delta(RG)\), and suppose \(G\) is a \(2\)-group. Then, \(G\) has exponent \(2\).
\end{lemma}

\begin{proof}
We first show that, for any \(g \in G\) and any \(k \in \mathbb{N}\), we have \(1 + g^{2^k} \in U(RG)\). Since \(R\langle g \rangle\) is a good subring of \(RG\), Lemma~\ref{goodsubring} tells us that \(R\langle g \rangle\) is also a \(W\Delta U\) ring. Thus, we may assume without loss of generality that the element \(g\) is central lying in the center. So, because \(G\) is a \(2\)-group, there is a minimal positive integer \(k\) such that \(g^{2^k} = 1\). Therefore, \(1 - g^{2^k} = 0\). Next, using the hypothesis \(3 \in \Delta(RG)\), we obtain
	\[
	1 + g^{2^k} = 1 - g^{2^k} - g^{2^k} + 3g^{2^k} = -g^{2^k} + 3g^{2^k} \in U(RG) + \Delta(RG) \subseteq U(RG).
	\]
Now, consider the product \((1 - g^{2^{k-1}})(1 + g^{2^{k-1}}) = 0\). Since \(1 + g^{2^{k-1}} \in U(RG)\), we must have \(1 - g^{2^{k-1}} = 0\), i.e., \(g^{2^{k-1}} = 1\), thus contradicting the minimality of \(k\). Hence, such a \(k\) cannot exist unless \(k = 1\), which means that \(g^2 = 1\) for all \(g \in G\). That is why, \(G\) is of exponent \(2\) after all.
\end{proof}

\bigskip

\noindent{\bf Acknowledgement.} This work is based upon research funded by Iran National Science Foundation (INSF) under project No. 40402401.


\vskip4.0pc

\begin{thebibliography}{99}
	
\bibitem{16}
H. Chen, {\it On strongly J-clean rings}, Commun. Algebra {\bf 38} (2010), 3790--3804.	

\bibitem{15}
W. Chen, {\it On constant products of elements in skew polynomial rings}, Bull. Iran. Math. Soc. {\bf 41}(2) (2015), 453--462.

\bibitem{cc}
W. Chen and S. Cui, {\it On clean rings and clean elements}, Southeast Asian Bull Math {\bf 32} (2008), 855--861.

\bibitem{ch}
H. Chen and M. Sheibani, {\it Strongly weakly nil-clean rings}, J. Algebra Appl. {\bf 16} (2017).

\bibitem{con}
I. G. Connell, {\it On the group ring}, Can. J. Math. {\bf 15} (1963), 650--685.

\bibitem{D}
P. V. Danchev, {\it Rings with Jacobson units}, Toyama Math. J. {\bf 38}(1) (2016), 61--74.

\bibitem{wuu}
P. V. Danchev, {\it Weakly UU rings}, Tsukuba J. Math. {\bf 40} (2016), 101--118.

\bibitem{wuj}
P. V. Danchev, {\it Weakly JU rings},  Missouri J. Math. Sci. {\bf 29}(2) (2017), 184--196.

\bibitem{Dj}
P. Danchev, A. Javan, O. Hasanzadeh and A. Moussavi, {\it Rings in which all elements are the sum of a central element and an element from $\Delta(R)$}, Izv. Saratov Univ. (Math., Mech. \& Inform.) {\bf 26}(2) (2026).

\bibitem{12}
P. V. Danchev and T. Y. Lam, {\it Rings with unipotent units}, Publ. Math. (Debrecen) {\bf 88}(3-4) (2016), 449--466.

\bibitem{7}
F. Karaba\c{c}ak, M. T. Kosan, T. Quynh and D. Tai, {\it A generalization of UJ-rings}, J. Algebra Appl. {\bf 20} (2021).

\bibitem{14}
M. T. Kosan, A. Leroy and J. Matczuk, {\it On UJ-rings}, Commun. Algebra {\bf 46}(5) (2018), 2297--2303.

\bibitem{9}
M. T. Kosan, S. Sahinkaya and Y. Zhou, {\it On weakly clean rings}, Commun. Algebra {\bf 45}(8) (2017), 3494--3502.

\bibitem{23}
T. Y. Lam, {\it Exercises in Classical Ring Theory}, Springer-verlag, New York (2003), second edition.

\bibitem{lam}
T. Y. Lam, {\it A First Course in Noncommutative Rings}, Springer-verlag, New York (2001), second edition.
	
\bibitem{2}
A. Leroy and J. Matczuk, {\it Remarks on the Jacobson radical rings}, Modules and Codes, Contemp. Math. {\bf 727} (2019), 269--276.

\bibitem{3}
J. Levitzki, {\it On the structure of algebraic algebras and related rings}, Trans. Am. Math. Soc. {\bf 74} (1953), 384--409.

\bibitem{Mari}
M. Marianne, {\it Rings of quotients of generalized matrix rings}, Commun. Algebra {\bf 15}(10) (1987), 1991--2015.

\bibitem{8}
W. K. Nicholson, {\it Lifting idempotents and exchange rings}, Trans. Am. Math. Soc. {\bf 229} (1977), 269--278.

\bibitem{25}
W. K. Nicholson and Y. Zhou, {\it Clean general rings}, J. Algebra {\bf 291}(1) (2005), 297--311.

\bibitem{wang}
W. Wang, E.R. Puczylowski, L. Li {\it On Armendariz rings and matrix rings with simple 0-multiplication}, Commun. Algebra {\bf 36}(4) (2008), 1514--1519.

\bibitem{26}
Y. Zhou, {\it On clean group rings}, Advances in Ring Theory, Trends in Mathematics, Birkhauser, Verlag Basel/Switzerland, 2010, pp. 335--345.

\end{thebibliography}
\end{document}